\documentclass[reqno]{amsart}%
\usepackage{amsfonts}
\usepackage{verbatim}
\usepackage{xcolor}
\usepackage{amsmath}
\usepackage{amssymb}
\usepackage{graphicx}
\usepackage{accents}%
\providecommand{\U}[1]{\protect\rule{.1in}{.1in}}
\newtheorem{theorem}{Theorem}[section]

\newtheorem{definition}[theorem]{Definition}

\newtheorem{remark}[theorem]{Remark}
\newtheorem{lemma}[theorem]{Lemma}
\newtheorem{example}[theorem]{Example}
\newtheorem{hypothesis}[theorem]{Hypothesis}
\numberwithin{equation}{section}
\begin{document}
\title[KdV equation]{The Inverse Scattering Transform for the KdV equation with step-like singular
Miura initial profiles}
\author{Sergei Grudsky}
\address{Departamento de Matematicas, CINVESTAV del I.P.N. Aportado Postal 14-740,
07000 Mexico, D.F., Mexico}
\email{grudsky@math.cinvestav.mx}
\author{Christian Remling}
\address{Department of Mathemics, University of Oklahoma, Norman, OK 73019}
\email{cremling@math.ou.edu}
\author{Alexei Rybkin}
\address{Department of Mathematics and Statistics, University of Alaska Fairbanks, PO
Box 756660, Fairbanks, AK 99775}
\email{arybkin@alaska.edu}
\thanks{SG is supported by PROMEP (M\'{e}xico) via "Proyecto de Redes", by CONACYT
grant 180049, and by Federal program \textquotedblleft Scientific and
Scientific-Pedagogical Personnel of Innovative Russia for the years
2007-2013\textquotedblright\ (contract No 14.A18.21.0873) }
\thanks{CR is supported in part by the NSF grant DMS 1200553 .}
\thanks{AR is supported in part by the NSF grant DMS 1009673.}
\date{November, 2014}

\begin{abstract}
We develop the inverse scattering transform for the KdV equation with real
singular initial data $q\left(  x\right)  $ of the form $q\left(  x\right)
=r^{\prime}\left(  x\right)  +r\left(  x\right)  ^{2}$, where $r\in
L_{loc}^{2}$, $r|_{\mathbb{R}_{+}}=0$. As a consequence we show that the
solution $q\left(  x,t\right)  $ is a meromorphic function with no real poles
for any $t>0$.

\end{abstract}
\maketitle
\tableofcontents

\section{Introduction}

This note is motivated by the recent progress in the spectral theory of
Schr\"{o}dinger operators with singular potentials
\cite{EckhartGesztesyetal2013} and the long lasting interest in completely
integrable systems with low regularity initial data \cite{KapPerryTopalov2005}.

We are concerned with the Cauchy problem for the Korteweg-de Vries (KdV)
equation ($x\in\mathbb{R},t>0$)%
\begin{equation}%
\begin{cases}
\partial_{t}u-6u\partial_{x}u+\partial_{x}^{3}u=0\\
u(x,0)=q(x)
\end{cases}
\label{KdV}%
\end{equation}
with the initial data $q$ satisfying (the terminology will be clarified below)

\begin{hypothesis}
\label{hyp1.1}$q\left(  x\right)  $ is a real-valued $H_{\operatorname*{loc}%
}^{-1}\left(  \mathbb{R}\right)  $ distribution subject to

\begin{enumerate}
\item[(1)] (positivity)%
\begin{equation}
\mathbb{L}_{q}\geq0 \label{Cond1}%
\end{equation}

\item[(2)] (restricted support)%
\begin{equation}
q|_{\mathbb{R}_{+}}=0. \label{Cond2}%
\end{equation}

\end{enumerate}
\end{hypothesis}

Here%
\[
H_{\operatorname*{loc}}^{-1}\left(  \mathbb{R}\right)  =\left\{  \chi
f:\chi\in C_{0}^{\infty}\left(  \mathbb{R}\right)  ,f\in H^{-1}\left(
\mathbb{R}\right)  \right\}  ,
\]
where $H^{s}\left(  \mathbb{R}\right)  $, $s\in\mathbb{R}$, is the Sobolev
space of distributions subject to $(1+\left\vert x\right\vert )^{s}%
\widehat{f}(x)\in L^{2}\left(  \mathbb{R}\right)  $ and $C_{0}^{\infty}\left(
\mathbb{R}\right)  $ is the space of compactly supported smooth functions.%
\begin{equation}
\mathbb{L}_{q}=-\partial_{x}^{2}+q(x) \label{schrodinger}%
\end{equation}
is the Schr\"{o}dinger operator on $L^{2}\left(  \mathbb{R}\right)  $
associated with the initial profile $q$ in (\ref{KdV}). As shown in
\cite{SS1999} $\mathbb{L}_{q}$ is well-defined as a selfadjoint operator for
large classes of distributional potentials $q$ form $H_{\operatorname*{loc}%
}^{-1}\left(  \mathbb{R}\right)  $. Condition (\ref{Cond1}) is understood in
the sense that $\langle\mathbb{L}_{q}\chi,\chi\rangle\geq0$ if $\chi\in
C_{0}^{\infty}\left(  \mathbb{R}\right)  $. It is one of the main results of
\cite{KapPerryTopalov2005} that Condition (\ref{Cond1}) holds if and only if%
\begin{equation}
q\left(  x\right)  =\partial_{x}r\left(  x\right)  +r\left(  x\right)
^{2}=:B\left(  r\right)  \label{Miura transform}%
\end{equation}
with some real $r\in L_{\operatorname*{loc}}^{2}\left(  \mathbb{R}\right)  $.
The transform (potential) $B\left(  r\right)  $ is referred in
\cite{KapPerryTopalov2005} to as Miura. To reflect Condition (\ref{Cond2}) we
call any $q$ subject to Hypothesis \ref{hyp1.1} a Miura steplike potential.

We emphasize from the beginning that our results can be suitably adjusted to
semiboundedness from below in Condition (\ref{Cond1}) and a certain decay
assumption at $+\infty$ in Condition (\ref{Cond2}). The numerous complications
(some of which are by no means trivial) that arise then are not of principal
nature and can all be resolved within our approach. They however seriously
aggravate the exposition. We therefore choose here transparency over completeness.

Our main goal is to develop the Inverse Scattering Transform (IST) method for
(\ref{KdV}) under Hypothesis \ref{hyp1.1}. We achieve our goal by employing
techniques of Hankel operators from \cite{GruRyPAMS13}, \cite{RybCommPDEs2013}%
-\cite{Ryb10}. The version of this approach that we use here makes our
considerations particularly transparent. More specifically, with an initial
profile subject to Hypothesis \ref{hyp1.1} we associate the Hankel operator
(see Definition \ref{def4.1}) $\mathbb{H}(\varphi_{x,t})$ with the symbol
$\varphi_{x,t}(k)=\xi_{x,t}(k)R(k)$ where%

\begin{equation}
\xi_{x,t}(k):=\exp\{i(8k^{3}t+2kx)\} \label{cubic exp}%
\end{equation}
and $R(k)$ is the reflection coefficient from the right incident (see
Definition \ref{def R}). The solution to (\ref{KdV}) is then given by%
\begin{equation}
q(x,t)=-2\partial_{x}^{2}\log\det\left(  1+\mathbb{H}(\varphi_{x,t})\right)
\label{solution}%
\end{equation}
thus establishing well-posedness of (\ref{KdV}) in the sense of Definition
\ref{WP}. Moreover, $q(x,t)$ is a meromorphic function in $x$ on the entire
complex plane for any $t>0$ with no real poles. We prove (\ref{solution}) by
first approximating our singular $q$ by $C_{0}^{\infty}$-functions for which
(\ref{solution}) is well-known and then passing to the limit. Justifying the
validly of our limiting arguments is the main issue here and it is the
techniques of Hankel operators that make it quite effortless.

Let us now put our results in the historic context. The formula
(\ref{solution}) is a derivation of the classical Dyson (also called Bargman
or log-determinant) formula (see, e.g. \cite{Dyson76}, \cite{Popper84}). For
step like (regular) potentials it appeared first in \cite{Venak86} (under
assumption that $q\left(  x\right)  $ goes to a constant at $-\infty$) and
with no restrictions on $-\infty$ in \cite{GruRyPAMS13},
\cite{RybCommPDEs2013}-\cite{Ryb10}. For singular potentials (\ref{solution})
is new. In fact, to the best of our knowledge, in the context of singular
initial data, the IST is rigorously justified for measure potentials (see e.g.
\cite{Kap1986}). On the other hand, well-posedness of (\ref{KdV}) in the
Sobolev space $H^{s}\left(  \mathbb{R}\right)  $ with negative index $s$
turned out to be an interesting problem in its own right \ having drawn
enormous attention. The sharpest result says that (\ref{KdV}) is globally
well-posed in $H^{-3/4}\left(  \mathbb{R}\right)  $ (\cite{ColKeStaTao03},
\cite{Guo09}, \cite{Tao06} and extensive literature therein). Note that the
space $H^{-3/4}\left(  \mathbb{R}\right)  $ includes such singular functions
as $\delta\left(  x\right)  $, $1/x,$ etc. However, $s=-3/4$ is the threshold
for the harmonic analytical methods commonly used in this circle of issues. On
the other hand, if one looks at the KdV as a completely integrable system the
Schr\"{o}dinger operator (\ref{schrodinger}) in the Lax pair associated with
(\ref{KdV}) remains well-defined for $s<-3/4$. In fact, the spectral (direct
and inverse) theory of Schr\"{o}dinger (Sturm-Liouville) operators with
singular potentials has independently attracted much of interest. The
systematic approach to $H_{\operatorname*{loc}}^{-1}\left(  \mathbb{R}\right)
$ potentials began with the influential paper \cite{SS1999} and has
experienced a rapid development culminating in the recent
\cite{EckhartGesztesyetal2013}, where the completeness of this theory
approaches that of the classical Titchmarsh-Weyl theory. We especially mention
the recent \cite{Hrynivetal2011} devoted to the scattering theory for
potentials $q\in B(L^{2}\left(  \mathbb{R}\right)  \cap L^{1}\left(
\mathbb{R}\right)  )$ with the future goal of developing the IST for such
initial data\footnote{To the best of our knowledge this goal has not been
realized yet.}. Our methods are completely different.

This suggests that if we use complete integrability of (\ref{KdV}) the global
well-posedness could be pushed across the threshold $s=-3/4$. It is exactly
how Kappeler-Topalov \cite{KapTop06} were able to extend well-posedness of
$H^{-1}(\mathbb{T})$ for periodic $q$'s. One might conjecture that the global
well-posedness for (\ref{KdV}) also holds far beyond $H^{-3/4}\left(
\mathbb{R}\right)  $ and could be achieved by a suitable extension of the IST
method for $\mathbb{L}_{q}$ with $q\in H^{-1}\left(  \mathbb{R}\right)  $. An
important step in this direction was done by Kappeler et al
\cite{KapPerryTopalov2005} where it was shown that (\ref{KdV}) is globally
well-posed in a certain sense if $q\in B(L^{2}\left(  \mathbb{R}\right)  )$.
Of course, $B(L^{2}\left(  \mathbb{R}\right)  )$ doesn't exhaust
$H^{-1}\left(  \mathbb{R}\right)  $ but, since $H^{-1}\left(  \mathbb{R}%
\right)  \supset H^{-3/4}\left(  \mathbb{R}\right)  $, singularity of such
solutions is pushed all the way to\footnote{As indicated in \cite{SS1999},
$\mathbb{L}_{q}$ with $q\in H^{-s}$ for $s>1$ is ill-defined.} $s=-1$.

We note that all functions in $H^{-s}(\mathbb{R})$ exhibit certain decay at
$\pm\infty$. On the other hand, there has been a significant interest in
non-decaying solutions to (\ref{KdV}) (other than periodic). The case of the
so-called steplike initial profiles (i.e. when $q(x)\rightarrow0$ sufficiently
fast as $x\rightarrow+\infty$ and $q(x)$ doesn't decay at $-\infty$) is of
physical interest and has attracted much attention since the early 70s. We
refer to the recent paper \cite{EGT09} for a comprehensive account of the
(rigorous) literature on steplike initial profiles with specified behavior at
infinity (e.g. $q$'s tending to a constant, periodic function, etc.). In the
recent \cite{GruRyPAMS13}, \cite{RybCommPDEs2013}-\cite{Ryb10} the case of
$q$'s rapidly decaying at $+\infty$ and sufficiently arbitrary at $-\infty$ is
studied in great detail. Initial steplike profiles in these papers are at
least locally integrable (i.e. regular).

The paper is organized as follows. In Section 2 we discuss the Titchmarsh-Weyl
$m$-function and reflection coefficient in the context of singular potentials.
In Section 3 we review Hankel operators and prove some results related to a
Hankel operator with a cubic oscillatory symbol. In the last Section 4 we
state and prove our main result.

\section{The Titchmarsh-Weyl $m$-function and the reflection coefficient}

It is well-known \cite{SS1999} that any $q\in H_{\operatorname*{loc}}%
^{-1}(\mathbb{R})$ can be represented as $q=\partial_{x}Q$ with some $Q\in
L_{\operatorname*{loc}}^{2}(\mathbb{R})$. We now regularize the
Schr\"{o}dinger differential expression with (formal) potential $q=\partial
_{x}Q$ by introducing the quasi-derivative%
\[
Dy=\partial_{x}y-Qy,
\]
for (locally) absolutely continuous $y$.

Following the approach of \cite{SS1999} we then introduce%
\begin{equation}%
\begin{cases}
\mathbb{L}_{q}y:=-\partial_{x}(Dy)-Q\partial_{x}y\\
\partial_{x}Q=q
\end{cases}
\label{eq3.1}%
\end{equation}
the Schr\"{o}dinger operator with a (singular) potential $q\in
H_{\operatorname*{loc}}^{-1}(\mathbb{R})$. This may be evaluated on functions
$y$ that satisfy $y,Dy\in AC$. Similarly, we regularize the Schr\"{o}dinger
equation by rewriting it as follows:
\begin{equation}
-\partial_{x}(Dy)-Q\partial_{x}y=zy. \label{eq3.0}%
\end{equation}

As proven in \cite{SS1999}, the operator \eqref{eq3.1} is well-defined.
Similarly, the classical Titchmarsh-Weyl theory can be developed for
$\mathbb{L}_{q}$ along the usual lines \cite{EckhartGesztesyetal2013}.
Essentially, one has to replace regular derivatives $\partial_{x}$ by the
quasi-derivative $D$ where appropriate.

Let us make this more explicit. Rewrite \eqref{eq3.0} as the first order
system
\begin{equation}
\partial_{x}Y=AY,\quad A=%
\begin{pmatrix}
Q & 1\\
-z-Q^{2} & -Q
\end{pmatrix}
,\ \ Y=%
\begin{pmatrix}
y\\
Dy
\end{pmatrix}
. \label{eq3.21}%
\end{equation}
Notice that $\operatorname*{tr}A=0$, so the modified Wronskian $W=y_{1}%
Dy_{2}-(Dy_{1})y_{2}$ of two solutions to the same equation is independent of
$x$. In particular, that means that the transfer matrices $T(x,z)$ associated
with \eqref{eq3.21} take values in $SL(2,\mathbb{C})$ and if $z\in\mathbb{R}$,
then $T(x,z)\in SL(2,\mathbb{R})$. Here we define $T$ as usual as the
$2\times2$ matrix solution of \eqref{eq3.21} with the initial value
$T(0,z)=1$. Moreover, if $z\in\mathbb{C}^{+}$, then $T(x,z)$ acting as a
linear fractional transformation on $w\in\mathbb{C}^{+}$ is a Herglotz
function, that is, the map $w\mapsto Tw$ maps $\mathbb{C}^{+}$ holomorphically
to itself. Recall also that this action is defined as
\[%
\begin{pmatrix}
a & b\\
c & d
\end{pmatrix}
w=\frac{aw+b}{cw+d}.
\]
To establish this Herglotz property of $w\mapsto T(x,z)w$, use the fact that
this property is equivalent to%
\[
-i(T^{\ast}JT-J)\geq0,\text{ where }J=\left(
\begin{smallmatrix}
0 & -1\\
1 & 0
\end{smallmatrix}
\right)  ;
\]
see, for example, \cite[Lemma 4.2]{Rem}. Now this latter property follows from
the fact that $-i(JA-(JA)^{\ast})\geq0$; indeed, a calculation shows that this
latter matrix equals $2\operatorname{Im}z\left(
\begin{smallmatrix}
1 & 0\\
0 & 0
\end{smallmatrix}
\right)  $.

These properties are all we need to have a Titchmarsh-Weyl type theory
available. See again \cite[Section 4]{Rem} for more on this abstract
interpretation of the theory.

We now define the Titchmarsh-Weyl $m$ function of the problem on $(-\infty
,0)$, with Dirichlet boundary conditions $y(0)=0$, as follows: For
$z\in\mathbb{C}^{+}$, let $\psi(\cdot,z)$ be the (unique, up to a factor)
solution of \eqref{eq3.0} that is square integrable near $-\infty$, and let%
\begin{equation}
m(z)=-\frac{D\psi(0,z)}{\psi(0,z)}. \label{eq4.1}%
\end{equation}
Then $m$ is a Herglotz function. Moreover, we have continuous dependence on
the potential, in the following sense.

\begin{theorem}
\label{m-funct conv}Let $Q_{n},Q\in L_{\operatorname*{loc}}^{2}(\mathbb{R})$.
Suppose that $\mathbb{L}_{q}$ is in the limit point case at $-\infty$. If
$Q_{n}\rightarrow Q$ in $L_{\operatorname*{loc}}^{2}(\mathbb{R})$, that is,
$\Vert Q_{n}-Q\Vert_{L^{2}(-R,R)}\rightarrow0$ for all $R>0$, then ${m}%
_{n}\rightarrow m$ uniformly on compact subsets of $\mathbb{C}^{+}$.
\end{theorem}

We do not assume limit point case for the operators $\mathbb{L}_{q_{n}}$
here\footnote{In fact, under Hypothesis \ref{hyp1.1} $q_{n}$ and $q$ are limit
point case (\cite{AlKoMa JMP 10}, \cite{EckhartGesztesyetal2013}).}. If some
or all of these are in the limit circle case, then we can make an arbitrary
choice of boundary conditions at $-\infty$ in the Theorem \ref{m-funct conv}.
As the proof below will make clear, what happens far out is in fact quite irrelevant.

\begin{proof}
By our limit point assumption, $m(z)$ can be approximated locally uniformly by
$m$ functions $m_{L}$ of problems on $[-L,0]$, with Dirichlet boundary
conditions at $x=-L$ (say), if we send $L\to\infty$ here. This follows from
the fact that such $m$ functions lie in the corresponding Weyl disks
$D_{L}(z)$ whose radii go to zero as $L\to\infty$, locally uniformly on
$z\in\mathbb{C}^{+}$. This last statement can be obtained in a general version
from a normal families argument (see \cite[Theorem 4.4]{Rem} for such a
treatment), or one can use, in more classical style, an explicit formula for
the radius of the Weyl disk in terms of entries of the transfer matrix.

Now we have that $m_{L}(z)=-(Dy_{L})(0)/y_{L}(0)$, where $y_{L}$ solves
\eqref{eq3.0} and $y_{L}(-L)=0$, $(Dy_{L})(-L)=1$. We also know that
$y_{L}(0,z)$ is bounded away from zero, uniformly on compact subsets of
$\mathbb{C}^{+}$. It now suffices to show that the values $y_{L}(0),
(Dy_{L})(0)$ that are obtained by solving \eqref{eq3.21} across $[-L,0]$
depend continuously on $Q$ in the sense specified, locally uniformly in $z$.
This is done by a rather routine argument; the key feature that makes things
work is the continuous dependence of $A(Q)$ on $Q$ in the $L^{1}$ norm on
$[-L,0]$. We include a sketch of the argument for the reader's convenience.

Write \eqref{eq3.21} in integral form:%
\begin{equation}
Y_{n}(x,z)=%
\begin{pmatrix}
0\\
1
\end{pmatrix}
+\int_{-L}^{x}A(t;Q_{n}(t),z)Y_{n}(t,z)\,dt \label{eq3.22}%
\end{equation}
We want to show that $Y_{n}(0,z)\rightarrow Y(0,z)$, where $Y$ solves the same
equation for $Q$. First of all, by standard ODE theory, the $Y_{n}$ are
uniformly bounded, that is, $|Y_{n}(x,z)|\leq C$ for $n\geq1$, $-L\leq x\leq
0$, $z\in K\subset\mathbb{C}$, and here $|\cdot|$ denotes an arbitrary norm on
$\mathbb{C}^{2}$. This implies that for $s<t$\footnote{We write $y\lesssim x$
in place of $y\leq Cx$ with some $C>0$ independent of $x$.},
\[
|Y_{n}(s,z)-Y_{n}(t,z)|\lesssim\int_{s}^{t}|A(t;Q_{n}(t),z)|\,dt\rightarrow
\int_{s}^{t}|A(t;Q(t),z)|\,dt.
\]
Here we have used the crucial fact that $A(t;Q_{n})\rightarrow A(t;Q)$ in
$L^{1}(I)$ for all compact intervals $I$. The convergence on the right-hand
side is uniform in $-L\leq s\leq t\leq0$ and $z\in K$; moreover $\int_{s}%
^{t}|A(Q)|$ can be made arbitrarily small by taking $|t-s|<\epsilon$. We have
verified that $Y_{n}(\cdot,z)$ is an equicontinuous family. Thus we may pass
to a limit in \eqref{eq3.22} along a suitable subsequence. It is easy to
verify that the integrals on the right-hand side also approach the expected
limit. So, if we write $Z=\lim Y_{n_{j}}$, then we obtain that
\[
Z(x,z)=%
\begin{pmatrix}
0\\
1
\end{pmatrix}
+\int_{-L}^{x}A(t;Q(t),z)Z(t,z)\,dt.
\]
This identifies $Z=Y$ as the solution of \eqref{eq3.21} for $Q$; also, since
this is the only possible limit, it was not necessary to pass to a
subsequence. In particular, we have established that $Y_{n}(0,z)\rightarrow
Y(0,z)$, as desired.
\end{proof}

\begin{remark}
For regular potentials Theorem \ref{m-funct conv} is a folklore. For singular
$H_{\operatorname*{loc}}^{-1}(\mathbb{R})$ potentials it is new.
\end{remark}

Define now the reflection coefficient $R$ from the right incident of a
singular potential $q\in H_{\operatorname*{loc}}^{-1}(\mathbb{R})$ such that
$q|_{\mathbb{R}_{+}}=0$. Note that for such a $q$, we may alternatively
compute $m$ as
\[
m(z)=-\lim_{x\rightarrow0+}\frac{\partial_{x}\psi(x,z)}{\psi(x,z)},
\]
and we can similarly replace $D\psi$ with $\partial_{x}\psi$ in $m$ functions
of problems on $(-\infty,x_{0})$, with $x_{0}>0$.

Pick a point $x_{0}>0$ and consider a solution to $\mathbb{L}_{q}y=\lambda
^{2}y$ which is proportional to the Weyl solution on $(-\infty,x_{0})$ and is
equal to $e^{-i\lambda x}+re^{i\lambda x}$ on $(x_{0},\infty)$. From the
continuity of this solution and its derivative at $x_{0}$ one has%
\[
r(\lambda,x_{0})=e^{-2i\lambda x_{0}}\frac{i\lambda-\dfrac{\psi^{\prime}%
(x_{0},\lambda^{2})}{\psi(x_{0},\lambda^{2})}}{i\lambda+\dfrac{\psi^{\prime
}(x_{0},\lambda^{2})}{\psi(x_{0},\lambda^{2})}}.
\]

We define the right reflection coefficient by

\begin{definition}
[Reflection coefficient]\label{def R}We call%
\begin{equation}
R(\lambda)=\lim_{x_{0}\rightarrow0^{+}}r(\lambda,x_{0})=\frac{i\lambda
-m(\lambda^{2})}{i\lambda+m(\lambda^{2})}.\label{eq8.2}%
\end{equation}
the (right) reflection coefficient.
\end{definition}

Theorem \ref{m-funct conv} and Definition \ref{def R} immediately imply

\begin{theorem}
\label{props of R}Assume that $q,q_{n}$ are subject to Hypothesis
\ref{hyp1.1}. Then $R\in H^{\infty}\left(  \mathbb{C}^{+}\right)  ,$%
\[
R(-\overline{\lambda})=\overline{R(\lambda)},\ \ \left\vert R\left(
\lambda\right)  \right\vert \leq1\text{ \ for\  }\lambda\in\mathbb{C}^{+},
\]
and  ${R}_{n}\left(  \lambda\right)  \rightarrow R\left(  \lambda\right)  $ as
$n\rightarrow\infty$ uniformly on compact subsets of $\mathbb{C}^{+}$ if
$q_{n}\rightarrow q$ in $H_{\operatorname*{loc}}^{-1}(\mathbb{R})$.
\end{theorem}

Note that Hypothesis \ref{hyp1.1} does not rule out the case $\left\vert
R\left(  \lambda\right)  \right\vert =1$ a. e. on the real line (in the
contrast with the short range case when $\left\vert R\left(  \lambda\right)
\right\vert <1$ for $\lambda\neq0$). From the spectral point of view the
latter means that the absolutely continuous spectrum of $\mathbb{L}_{q}$ is
supported on $\mathbb{R}_{+}$ but has uniform multiplicity one (not two as in
the short range case). We conclude the section with an explicit example.

\begin{example}
Let $q(x)=c\delta(x),c>0$. The Weyl solution corresponding to $-\infty$ can be
explicitly computed by ($C\neq0$)
\[
\psi(x,\lambda^{2})=C%
\begin{cases}
e^{-i\lambda x}\quad & ,\quad x<0\\
\frac{1}{2i\lambda}\left(  ce^{i\lambda x}+(2i\lambda-c)e^{-i\lambda
x}\right)  \quad & ,\quad x>0
\end{cases}
\]
and hence by \eqref{eq4.1} and \eqref{eq8.2}
\[
m(\lambda^{2})=i\lambda-c,\ \ \ R(\lambda)=\frac{c}{2i\lambda-c}.
\]

\end{example}

\section{Hankel Operators\label{hankel}}

A Hankel operator is an infinitely dimensional analog of a Hankel matrix, a
matrix whose $(j,k)$ entry depends only on $j+k$. I.e. a matrix $\Gamma$ of
the form
\[
\Gamma=\left(
\begin{array}
[c]{cccc}%
\gamma_{1} & \gamma_{2} & \gamma_{3} & ...\\
\gamma_{2} & \gamma_{3} & ... & \\
\gamma_{3} & ... &  & \\
... &  &  & \gamma_{n}%
\end{array}
\right)  .
\]
Definitions of Hankel operators depend on specific spaces. We consider Hankel
operators on the Hardy space $H^{2}\left(  \mathbb{C}^{+}\right)  $ (cf.
\cite{Nik2002}, \cite{Peller2003}). Here, as usual (but a bit in conflict with
our notation of the Sobolev spaces) $H^{p}\left(  \mathbb{C}^{\pm}\right)  $
($0<p\leq\infty$) denotes the Hardy space of $\mathbb{C}^{\pm}$. It is
well-known (see e.g.\cite{Garnett}) that $L^{2}\left(  \mathbb{R}\right)
=H^{2}\left(  \mathbb{C}^{+}\right)  \oplus H^{2}\left(  \mathbb{C}%
^{-}\right)  ,$ the orthogonal (Riesz) projection $\mathbb{P}_{\pm}$ onto
$H^{2}\left(  \mathbb{C}^{\pm}\right)  $ being given by%
\begin{equation}
(\mathbb{P}_{\pm}f)(x)=\pm\frac{1}{2\pi i}\int\frac{f(s)ds}{s-(x\pm i0)}.
\label{proj}%
\end{equation}
Let $(\mathbb{J}f)(x)=f(-x)$ be the operator of reflection on $L^{2}$. It is
clearly an isometry and with the obvious property%
\begin{equation}
\mathbb{JP}_{\mp}=\mathbb{P}_{\pm}\mathbb{J}. \label{eq4.9}%
\end{equation}

\begin{definition}
[Hankel and Toeplitz operators]\label{def4.1}Let $\varphi\in L^{\infty}\left(
\mathbb{R}\right)  $. The operators $\mathbb{H}(\varphi)$ and $\mathbb{T}%
(\varphi)$ defined by
\begin{equation}
\mathbb{H}(\varphi)f=\mathbb{JP}_{-}\varphi f,\;\text{and }\mathbb{T}%
(\varphi)f=\mathbb{P}_{+}\varphi f,\ \ f\in H^{2}\left(  \mathbb{C}%
^{+}\right)  , \label{Def H&T}%
\end{equation}
are called respectively the Hankel and Toeplitz operators with the symbol
$\varphi$.
\end{definition}

Due to (\ref{eq4.9}), both $\mathbb{H}(\varphi)$ and $\mathbb{T}(\varphi)$ act
from $H^{2}\left(  \mathbb{C}^{+}\right)  $ to $H^{2}\left(  \mathbb{C}%
^{+}\right)  $. Note that while $\mathbb{H}(\varphi)$ and $\mathbb{T}%
(\varphi)$ look alike, they are different parts of the multiplication
operator
\begin{equation}
\varphi f=\mathbb{JH}(\varphi)f+\mathbb{T}(\varphi)f,\ \;f\in H^{2}\left(
\mathbb{C}^{+}\right)  , \label{eq4.2}%
\end{equation}
and therefore are quite different. The Toeplitz operator will play only an
auxiliary role in our consideration.

As well-know (and also obvious) that $\mathbb{H}(\varphi)$ is selfadjoint if
$\mathbb{J}\varphi=\bar{\varphi}.$

Definition \ref{def4.1} can be extended to certain unbounded symbols {(more
exactly, from BMO}) which nevertheless produce bounded Hankel operators. In
such cases we define $\mathbb{H}(\varphi)$ first on the set
\begin{equation}
\mathfrak{H}_{2}:=\left\{  f\in H^{2}\left(  \mathbb{C}^{+}\right)  :f\in
C^{\infty},\ f\left(  z\right)  =o\left(  z^{-2}\right)  ,\ z\rightarrow
\infty,\ \operatorname{Im}z\geq0\right\}  , \label{dense set in H}%
\end{equation}
dense \cite{Garnett} in $H^{2}\left(  \mathbb{C}^{+}\right)  $ by%
\begin{equation}
\mathbb{H}(\varphi)f=\mathbb{JP}_{-}\varphi f,\ \ f\in\mathfrak{H}_{2},
\label{H on dense set}%
\end{equation}
and then extend (\ref{H on dense set})\ to the whole $H^{2}\left(
\mathbb{C}^{+}\right)  $ retaining the same notation $\mathbb{H}(\varphi)$ for
the extension.

Introduce the regularized Riesz projection%
\begin{equation}
(\widetilde{\mathbb{P}}_{\pm}f)(x)=\pm\frac{1}{2\pi i}\int_{\mathbb{R}}\left(
\frac{1}{s-(x\pm i0)}-\frac{1}{s+i}\right)  f(s)\ ds,\;\ \ f\in L^{\infty
}\left(  \mathbb{R}\right)  . \label{eq3.4}%
\end{equation}
As well-known%
\begin{equation}
\widetilde{\mathbb{P}}_{\pm}f\in\text{BMOA}(\mathbb{C}^{\pm})\text{ if }\ f\in
L^{\infty}\left(  \mathbb{R}\right)  , \label{BMOA}%
\end{equation}
where BMOA$(\mathbb{C}^{\pm})$ is the class of analytic in $\mathbb{C}^{\pm}$
functions having bounded mean oscillation:%
\[
\sup_{I\in\mathbb{R}}\frac{1}{|I|}\int_{I}|f\left(  x\right)  -f_{I}%
|\ dx<\infty,\ \ \ f_{I}:=\frac{1}{|I|}\int_{I}f\left(  x\right)  \ dx,
\]
for any bounded interval $I$. One has%
\begin{equation}
\widetilde{\mathbb{P}}_{+}f+\widetilde{\mathbb{P}}_{-}f=f,\;f\in L^{\infty
}\left(  \mathbb{R}\right)  . \label{eq3.5}%
\end{equation}
The next elementary statement will play a important role in our consideration.

\begin{theorem}
\label{th4.5}If $\varphi\in L^{\infty}\left(  \mathbb{R}\right)  $ then%
\begin{equation}
\mathbb{H}(\varphi)=\mathbb{H}(\widetilde{\mathbb{P}}_{-}\varphi).
\label{eq4.5}%
\end{equation}

\end{theorem}

\begin{proof}
As well-known \cite{Garnett} every BMOA\ function $h$ is subject to $h\left(
x\right)  /\left(  1+x^{2}\right)  \in L^{1}\left(  \mathbb{R}\right)  $ and
one can easily see that $hf\in H^{2}\left(  \mathbb{C}^{+}\right)  $ if
$f\in\mathfrak{H}_{2}$. Hence $\mathbb{P}_{-}hf=0$ and
\begin{equation}
\mathbb{H}(\varphi+h)f=\mathbb{H}(\varphi)f\text{ \ for any}\ \ f\in
\mathfrak{H}_{2}. \label{eq4.5'}%
\end{equation}
Therefore (\ref{eq4.5'}) can be closed to the whole $H^{2}\left(
\mathbb{C}^{+}\right)  $ and $\mathbb{H}(\varphi+h)$ is well-defined in the
sense discussed above, bounded and
\[
\mathbb{H}(\varphi+h)=\mathbb{H}(\varphi)
\]
holds. By (\ref{eq3.5}) $\varphi=\widetilde{\mathbb{P}}_{-}\varphi
+\widetilde{\mathbb{P}}_{+}\varphi$ and (\ref{eq4.5}) follows from
(\ref{eq4.5'}) with $h=-\widetilde{\mathbb{P}}_{+}\varphi$ which, by
(\ref{BMOA}), is in BMOA$\left(  \mathbb{C}^{+}\right)  $.
\end{proof}

Directly from Definition \ref{def4.1}, $\Vert\mathbb{H}(\varphi)\Vert\leq
\Vert\varphi\Vert_{\infty}$ but we will need much stronger statements.

Let us now introduce the Sarason algebra%
\[
H^{\infty}\left(  \mathbb{C}^{+}\right)  +C\left(  \mathbb{R}\right)
:=\{f:f=h+g,\;h\in H^{\infty}\left(  \mathbb{C}^{+}\right)  ,\;g\in C\left(
\mathbb{R}\right)  \},
\]
where%
\[
C\left(  \mathbb{R}\right)  =\left\{  f:f\text{ is continuous on }%
\mathbb{R}\text{, }\lim_{x\rightarrow\infty}f\left(  x\right)  =\lim
_{x\rightarrow-\infty}f\left(  x\right)  \neq\pm\infty\right\}  .
\]

\begin{theorem}
[Sarason, 1967]\label{sarason thm}$H^{\infty}\left(  \mathbb{C}^{+}\right)
+C\left(  \mathbb{R}\right)  $ is a closed sub-algebra of $L^{\infty}\left(
\mathbb{R}\right)  $.
\end{theorem}

The set $H^{\infty}+C$ is one of the most common function classes in the
theory of Hankel (and Toeplitz) operators due to the following fundamental theorem.

\begin{theorem}
[Hartman, 1958]\label{thHart}Let $\varphi\in L^{\infty}\left(  \mathbb{R}%
\right)  $. Then $\mathbb{H}(\varphi)$ is compact if and only if $\varphi\in
H^{\infty}\left(  \mathbb{C}^{+}\right)  +C\left(  \mathbb{R}\right)  $. I.e.
$\mathbb{H}(\varphi)$ is compact if and only if $\mathbb{H}(\varphi
)=\mathbb{H}(g)$ with some $g\in C\left(  \mathbb{R}\right)  $.
\end{theorem}

For Hankel operators appearing in completely integrable systems the membership
of the symbol in $H^{\infty}\left(  \mathbb{C}^{+}\right)  +C\left(
\mathbb{R}\right)  $ is far from being obvious. The following statement will
be crucial to our approach.

\begin{theorem}
[Grudsky, 2001]\label{thGru01}Let $p(x)$ be a real polynomial with a positive
leading coefficient such that%
\begin{equation}
p(-x)=-p(x). \label{eq5.1}%
\end{equation}
Then
\begin{equation}
e^{ip}\in H^{\infty}\left(  \mathbb{C}^{+}\right)  +C\left(  \mathbb{R}%
\right)  . \label{eq5.2}%
\end{equation}
Moreover, there exist an infinite Blaschke product $B$ and a unimodular
function $u\in C\left(  \mathbb{R}\right)  $ such that%
\begin{equation}
e^{ip}=Bu. \label{eq5.3}%
\end{equation}

\end{theorem}

In our case $p(\lambda)=t\lambda^{3}+x\lambda$ with real $x$ (spatial
variable) and positive $t$ (time). Note that for polynomials $p$ of even
order, Theorem \ref{thGru01} fails.

\begin{definition}
A function $f\in H^{\infty}\left(  \mathbb{C}^{+}\right)  +C\left(
\mathbb{R}\right)  $ is said invertible in $H^{\infty}\left(  \mathbb{C}%
^{+}\right)  +C\left(  \mathbb{R}\right)  $ if $1/f\in H^{\infty}\left(
\mathbb{C}^{+}\right)  +C\left(  \mathbb{R}\right)  $. Similarly, $f$ is not
invertible in $H^{\infty}\left(  \mathbb{C}^{+}\right)  +C\left(
\mathbb{R}\right)  $ if $1/f\notin H^{\infty}\left(  \mathbb{C}^{+}\right)
+C\left(  \mathbb{R}\right)  $.
\end{definition}

This concept is very important in the connection with invertibility of
Toeplitz operators, as the following theorem suggests (see, e.g.
\cite{BotSil06}, \cite{DybGru2002}).

\begin{theorem}
\label{th5.5}Let $\varphi\in H^{\infty}\left(  \mathbb{C}^{+}\right)
+C\left(  \mathbb{R}\right)  $ and $1/\varphi\in L^{\infty}\left(
\mathbb{R}\right)  $. Then%
\begin{align}
1/\varphi\notin H^{\infty}\left(  \mathbb{C}^{+}\right)  +C\left(
\mathbb{R}\right)   &  \Longrightarrow\mathbb{T}(\varphi)\text{ is
left-invertible,}\label{eq5.4}\\
1/\varphi\in H^{\infty}\left(  \mathbb{C}^{+}\right)  +C\left(  \mathbb{R}%
\right)   &  \Longrightarrow\mathbb{T}(\varphi)\;\text{is Fredholm.}
\label{eq5.5}%
\end{align}

\end{theorem}

\begin{lemma}
\label{lem5.6}Let $B$ be an infinite Blaschke product, $u\in H^{\infty}\left(
\mathbb{C}^{+}\right)  +C\left(  \mathbb{R}\right)  $ and unimodular. Then
$\varphi=Bu$ is not invertible in $H^{\infty}\left(  \mathbb{C}^{+}\right)
+C\left(  \mathbb{R}\right)  $.
\end{lemma}

\begin{proof}
(By contradiction). Since $B\in H^{\infty}\left(  \mathbb{C}^{+}\right)  $,
due to the algebraic property (Theorem \ref{sarason thm})\ of $H^{\infty
}\left(  \mathbb{C}^{+}\right)  +C\left(  \mathbb{R}\right)  ,\;$one has
$\varphi\in H^{\infty}\left(  \mathbb{C}^{+}\right)  +C\left(  \mathbb{R}%
\right)  $. Assume that $\varphi$ is invertible in $H^{\infty}\left(
\mathbb{C}^{+}\right)  +C\left(  \mathbb{R}\right)  $, i.e. $1/\varphi\in
H^{\infty}\left(  \mathbb{C}^{+}\right)  +C\left(  \mathbb{R}\right)  $. Then
by (\ref{eq5.5}) $\mathbb{T}(\varphi)$ is Fredholm that forces $\mathbb{T}(B)$
to be Fredholm too. Indeed, $B\in H^{\infty}\left(  \mathbb{C}^{+}\right)  $
and, since $\varphi=Bu$,%
\[
1/B=u\cdot\,1/\varphi\in H^{\infty}\left(  \mathbb{C}^{+}\right)  +C\left(
\mathbb{R}\right)  .
\]
Thus $B$ is invertible in $H^{\infty}\left(  \mathbb{C}^{+}\right)  +C\left(
\mathbb{R}\right)  $ and (\ref{eq5.5}) holds. Hence $\mathbb{T}(\overline
{B})=\mathbb{T}(1/B)$ is also Fredholm and therefore by definition%
\begin{equation}
\dim\ker\mathbb{T}(\overline{B})<\infty. \label{eq5.6}%
\end{equation}
We now show that (\ref{eq5.6}) may not hold for $B$ with infinitely many zeros
$\{z_{k}\}$, which creates a desired contradiction. To this end consider the
Blaschke product%
\[
B(x)=\prod b_{n}(x),\;b_{n}=c_{n}\left(  \frac{x-z_{n}}{x-\overline{z_{n}}%
}\right)
\]
and set%
\[
f_{n}(x)=c_{n}(x-\overline{z_{n}})^{-1}.
\]
Clearly $f_{n}\in H^{2}\left(  \mathbb{C}^{+}\right)  $ and
\[
\mathbb{T}(\overline{B})f_{n}=\mathbb{P}_{+}\overline{B}f_{n}=\mathbb{P}%
_{+}\overline{\overline{c_{n}}(\cdot-z_{n})^{-1}B}=\mathbb{P}_{+}(\cdot
-z_{n})^{-1}\overline{B_{n}},
\]
where $B_{n}=B/b_{n}$. But $\overline{B_{n}}\in H^{\infty}\left(
\mathbb{C}^{-}\right)  $ and $(x-z_{n})^{-1}\in H^{2}\left(  \mathbb{C}%
^{-}\right)  $. Hence
\[
(x-z_{n})^{-1}\overline{B_{n}(x)}\in H^{2}\left(  \mathbb{C}^{-}\right)
\]
and
\[
\mathbb{T}(\overline{B})f_{n}=0.
\]
Therefore $f_{n}\in\ker\mathbb{T}(\overline{B})$ and the lemma is proven as
$\{f_{n}\}$ are linearly independent.
\end{proof}

Note that Lemma \ref{lem5.6} is entirely about $H^{\infty}\left(
\mathbb{C}^{+}\right)  +C\left(  \mathbb{R}\right)  $ but its proof, as often
happens in this circle of issues, relies on operator theoretical arguments.

The next important claim directly follows from Theorem \ref{thGru01} and Lemma
\ref{lem5.6}.

\begin{theorem}
\label{theorem 5.4'}Let $u$ be a unimodular function from $H^{\infty}\left(
\mathbb{C}^{+}\right)  +C\left(  \mathbb{R}\right)  $ and $e^{ip}$ as in
Theorem \ref{thGru01}. Then $e^{ip}u$ is not invertible in $H^{\infty}\left(
\mathbb{C}^{+}\right)  +C\left(  \mathbb{R}\right)  $.
\end{theorem}

Combining Theorems \ref{th4.4} and

\begin{theorem}
[Widom, 1960]\label{th4.4}Let $\varphi$ be unimodular. Then $\Vert
\mathbb{H}(\varphi)\Vert<1$ if and only if $\mathbb{T}(\varphi)$ is left invertible.
\end{theorem}

yields

\begin{theorem}
\label{th5.7}If $\varphi\in H^{\infty}\left(  \mathbb{C}^{+}\right)  +C\left(
\mathbb{R}\right)  $ and unimodular but not invertible then%
\begin{equation}
\Vert\mathbb{H}(\varphi)\Vert<1. \label{eq5.7}%
\end{equation}

\end{theorem}

\begin{proof}
By Theorem \ref{th5.5}, $\mathbb{T}(\varphi)$ is left-invertible. By Theorem
\ref{th4.4} we have (\ref{eq5.7}).
\end{proof}

While an immediate consequence of Theorems \ref{th5.7} and \ref{thGru01}, the
following theorem is vital to our approach.

\begin{theorem}
\label{Thm 5.6'}If $u\in H^{\infty}\left(  \mathbb{C}^{+}\right)  +C\left(
\mathbb{R}\right)  $, $\left\vert u\right\vert =1$, and $p$ is as in Theorem
\ref{thGru01}, then%
\[
\Vert\mathbb{H}(e^{ip}u)\Vert<1.
\]

\end{theorem}

\begin{theorem}
\label{rem5.8} If $\varphi\in H^{\infty}\left(  \mathbb{C}^{+}\right)
+C\left(  \mathbb{R}\right)  $
is not unimodular but $\Vert\varphi\Vert_{\infty}\leq1$ and $\mathbb{J}%
\varphi=\bar{\varphi}$ then (\ref{eq5.7}) holds.
\end{theorem}

\begin{proof}
(By contradiction) Assume that $\Vert\mathbb{H}(\varphi)\Vert=1$. Since
$\mathbb{H}\left(  \varphi\right)  $ is selfadjoint and compact (by Theorems
\ref{thGru01} and \ref{thHart}), $\mathbb{H}(\varphi)$ has a unimodular
eigenvalue $\lambda$ ($\lambda=\pm1$). For the associated normalized
eigenfunction $f\in H^{2}\left(  \mathbb{C}^{+}\right)  $ we have by
(\ref{eq4.9})%
\[
\left\langle \mathbb{H}(\varphi)f,f\right\rangle =\left\langle \varphi
f,\mathbb{P}_{-}\mathbb{J}f\right\rangle =\left\langle \varphi f,\mathbb{J}%
f\right\rangle
\]
and hence by the Cauchy inequality%
\begin{align}
\left\vert \left\langle \mathbb{H}(\varphi)f,f\right\rangle \right\vert ^{2}
&  \leq\left(  \int_{\mathbb{R}}\left\vert \varphi\left(  x\right)  f\left(
x\right)  f\left(  -x\right)  \right\vert \ dx\right)  ^{2}\nonumber\\
&  \leq\int_{\mathbb{R}}\left\vert \varphi\left(  x\right)  \right\vert
\left\vert f\left(  x\right)  \right\vert ^{2}dx\ \int_{\mathbb{R}}\left\vert
f\left(  -x\right)  \right\vert ^{2}\ dx\nonumber\\
&  =\int_{\mathbb{R}}\left\vert \varphi\left(  x\right)  \right\vert
\left\vert f\left(  x\right)  \right\vert ^{2}\ dx\nonumber\\
&  =\int_{S}\left\vert \varphi\left(  x\right)  \right\vert \left\vert
f\left(  x\right)  \right\vert ^{2}\ dx+\int_{\mathbb{R}\diagdown S}\left\vert
\varphi\left(  x\right)  \right\vert \left\vert f\left(  x\right)  \right\vert
^{2}\ dx\label{eigen}\\
&  <\int_{\mathbb{R}}\left\vert \varphi\left(  x\right)  \right\vert
\left\vert f\left(  x\right)  \right\vert ^{2}\ dx=1.\nonumber
\end{align}
Here $S$ is a set of positive Lebesgue measure where $\left\vert
\varphi\left(  x\right)  \right\vert <1$ a.e. Here we have used the fact that
$f\in H^{2}\left(  \mathbb{C}^{+}\right)  $ and hence cannot vanish on $S$.
The inequality (\ref{eigen}) implies that $\left\vert \lambda\right\vert <1$
which is a contradiction.
\end{proof}

Finally, we note that the Hankel operator can also be defined as an integral
operator on $L^{2}(\mathbb{R}_{+})$ whose kernel depends on the sum of the
arguments%
\begin{equation}
(\mathbb{H}f)(x)=\int_{\mathbb{R}_{+}}h(x+y)f(y)dy,\;f\in L^{2}(\mathbb{R}%
_{+}),\;x\geq0 \label{eq4.10}%
\end{equation}
and it is this form that Hankel operators typically appear in the inverse
scattering formalism. One can show that the Hankel operator $\mathbb{H}$
defined by (\ref{eq4.10}) is unitary equivalent to $\mathbb{H}(\varphi)$ with
the symbol $\varphi$ equal to the Fourier transform of $h$. We emphasize
though that the form (\ref{eq4.10}) does not prove to be convenient for our
purposes and also $h$ is in general not a function but a distribution.

\section{The IST\ Hankel Operator\label{Our hankel}}

For a reason which will become clear in the next section we introduce

\begin{definition}
[IST Hankel operator]\label{def IST HO} Assume that initial data $q$ is
subject to Hypothesis \ref{hyp1.1}. Let $R$ be as in Definition \ref{def R}.
We call the Hankel operator
\[
\mathbb{H}(x,t):=\mathbb{H}(\varphi_{x,t}),
\]
with the symbol
\begin{equation}
\varphi_{x,t}(k)=\xi_{x,t}(k)R(k), \label{eq9.9}%
\end{equation}
the IST Hankel operator associated with $q$.
\end{definition}

We need some general statements on singular numbers of Hankel operators. We
recall that the $n$-th singular value $s_{n}\left(  \mathbb{A}\right)  $ of a
compact Hilbert space operator $\mathbb{A}$ is defined as the $n$-th
eigenvalue of the operator $\left(  \mathbb{A}^{\ast}\mathbb{A}\right)
^{1/2}$.

The following theorems are fundamental in the study of singular numbers of
Hankel operators.

\begin{theorem}
[Adamyan-Arov-Krein, 1971]\label{AAK} Let $\varphi\in L^{\infty}\left(
\mathbb{R}\right)  $. Then
\[
s_{n}\left(  \mathbb{H}\left(  \varphi\right)  \right)  =\operatorname*{dist}%
\nolimits_{L^{\infty}}\left(  \varphi,\mathcal{R}_{n}+H^{\infty}\left(
\mathbb{C}^{+}\right)  \right)  ,
\]
where $\mathcal{R}_{n\text{ }}$is the set of rational functions bounded at
infinity with all poles in $\mathbb{C}^{+}$ of total multiplicity $\leq n$.
\end{theorem}

\begin{theorem}
[Jackson, 1910]\label{Jackson}Let $\varphi\in C^{m}\left(  \mathbb{R}\right)
$. Then%
\[
\operatorname*{dist}\nolimits_{L^{\infty}}\left(  \varphi,\mathcal{R}%
_{n}+H^{\infty}\left(  \mathbb{C}^{+}\right)  \right)  \lesssim\left\Vert
\varphi^{\left(  m\right)  }\right\Vert _{\infty}/n^{m}.
\]

\end{theorem}

Theorems \ref{AAK} and \ref{Jackson} immediately yield the following observation.

\begin{lemma}
\label{prop on sn}Let $f\in L^{1}\left(  \mathbb{R}\right)  $, $h>0$ and%
\begin{equation}
\varphi(k)=\int_{\mathbb{R}}\frac{f(s)}{s-k+ih}ds. \label{small fi}%
\end{equation}
Then%
\[
s_{n}\left(  \mathbb{H}\left(  \varphi\right)  \right)  \lesssim\left(
2/h\right)  \left\Vert f\right\Vert _{1}\exp\left\{  -\left(  h/2\right)
n\right\}  .
\]

\end{lemma}

\begin{proof}
Differentiating (\ref{small fi}) one has%
\[
\left\Vert \varphi^{\left(  m\right)  }\right\Vert _{\infty}\leq\left\Vert
f\right\Vert _{1}\frac{m!}{h^{m+1}}%
\]
and hence by Theorems \ref{AAK} and \ref{Jackson} for any $m=0,1,2,...$%
\[
s_{n}\left(  \mathbb{H}\left(  \varphi\right)  \right)  \lesssim
\frac{\left\Vert f\right\Vert _{1}}{h}\frac{m!}{\left(  hn\right)  ^{m}}.
\]
Rewriting the last estimate as $s_{n}\left(  \mathbb{H}\left(  \varphi\right)
\right)  \dfrac{m!}{\left(  2hn\right)  ^{m}}\lesssim\dfrac{\left\Vert
f\right\Vert _{1}}{h}2^{-m}$ and summing up on $m$ implies the desired result.
\end{proof}

Here is the main result of this section

\begin{theorem}
[Properties of the IST Hankel operator]\label{th9.5} Under Hypothesis
\ref{hyp1.1} the IST Hankel operator $\mathbb{H}(x,t)$ is well-defined and has
the properties: for any $x\in\mathbb{R},\;t>0$

\begin{enumerate}
\item[(1)] $\mathbb{H}(x,t)$ is selfadjoint,

\item[(2)] $\mathbb{H}(x,t)$ is compact and its singular numbers $s_{n}\left(
\mathbb{H}\left(  x,t\right)  \right)  $ satisfy%
\begin{equation}
s_{n}\left(  \mathbb{H}\left(  x,t\right)  \right)  \lesssim\frac{2}%
{h}\left\{  \int_{\mathbb{R}}\left\vert \xi_{x,t}(\lambda+ih)R(\lambda
+ih)\right\vert d\lambda\right\}  \ \exp\left\{  -\frac{2n}{h}\right\}
\label{s-number est}%
\end{equation}
for any $h>0.$

\item[(3)] $\left\Vert \mathbb{H}(x,t)\right\Vert <1.$

\item[(4)] $\partial_{t}^{m}\mathbb{H}\left(  x,t\right)  ,\ m=0,1,$ is an
entire in $x$ operator-valued function.
\end{enumerate}
\end{theorem}

\begin{proof}
Statement (1) is obvious as $\mathbb{J}\varphi_{x,t}=\overline{\varphi_{x,t}}%
$. We now prove statement (2). By Theorem \ref{th4.5} ($\varphi_{x,t}%
=\xi_{x,t}R$)%
\[
\mathbb{H}(x,t)=\mathbb{H}(\varphi_{x,t})=\mathbb{H}(\widetilde{\mathbb{P}%
}_{-}\varphi_{x,t})
\]
where%
\begin{equation}
\left(  \widetilde{\mathbb{P}}_{-}\varphi_{x,t}\right)  (k)=-\frac{1}{2\pi
i}\int_{\mathbb{R}}\left(  \frac{1}{\lambda-(k-i0)}-\frac{1}{\lambda
+i}\right)  \varphi_{x,t}(\lambda)\ d\lambda. \label{principal part}%
\end{equation}
The function $\varphi_{x,t}$ is clearly analytic in $\mathbb{C}^{+}$. Since
$R\in H^{\infty}\left(  \mathbb{C}^{+}\right)  $ with $\left\Vert R\right\Vert
_{\infty}\leq1$ and $\xi_{x,t}(\lambda+ih)$ rapidly decays as $\lambda
\rightarrow\pm\infty$ for any $t>0$ and arbitrary $h>0$, we can deform the
contour of integration in (\ref{principal part}) to $R+ih,\ h>0$. Thus%
\[
\left(  \widetilde{\mathbb{P}}_{-}\varphi_{x,t}\right)  (k)=-\frac{1}{2\pi
i}\int_{\mathbb{R}+ih}\frac{\varphi_{x,t}(\lambda)}{\lambda-k}\ d\lambda
-\int_{\mathbb{R}+ih}\frac{\varphi_{x,t}(\lambda)}{\lambda+i}\ d\lambda.
\]
Since the last term is a constant, we conclude that%
\begin{equation}
\mathbb{H}(x,t)=\mathbb{H}\left(  \Phi_{x,t}\right)  \label{h=H}%
\end{equation}
with an entire function%
\begin{equation}
\Phi_{x,t}(k):=-\frac{1}{2\pi i}\int_{\mathbb{R}}\frac{\xi_{x,t}%
(\lambda+ih)R(\lambda+ih)}{\lambda-k+ih}\ d\lambda. \label{Phi}%
\end{equation}
By Theorem \ref{thHart} operator $\mathbb{H}(x,t)$ is compact. By Lemma
\ref{prop on sn} yields (\ref{s-number est}).

Statement (3) follows from Theorem \ref{rem5.8} if $\left\vert R\left(
\lambda\right)  \right\vert <1$ on a set of positive Lebesgue measure, or
Theorem \ref{Thm 5.6'} if $\left\vert R\left(  \lambda\right)  \right\vert =1$ a.e.

It remains to prove statement (4). One can easily see from the straightforward
formula ($t>0$)%
\begin{equation}
\left\vert \xi_{z,t}(\lambda+ih)\right\vert =\exp\left\{  8h^{3}%
t-2h\operatorname{Re}z+\frac{\operatorname{Im}^{2}z}{24ht}-\left(  \sqrt
{24ht}\lambda+\frac{\operatorname{Im}z}{\sqrt{24ht}}\right)  ^{2}\right\}
\label{ksi}%
\end{equation}
and (\ref{Phi}) that $\partial_{t}^{m}\Phi_{z,t}(k)$ is well-defined for any
complex $z$. I.e. it is also entire in $z$ for any $t>0$. Therefore the
operator-valued function $\partial_{t}^{m}\mathbb{H}\left(  x,t\right)  $
defined by%
\[
\partial_{t}^{m}\mathbb{H}\left(  x,t\right)  =\mathbb{H}\left(  \partial
_{t}^{m}\Phi_{x,t}\right)
\]
is also entire.
\end{proof}

We conclude this section with a few remarks.

\begin{remark}
\label{rem9.6}Statement (3) of Theorem \ref{th9.5} says that $(\mathbb{I}%
+\mathbb{H}(x,t))^{-1}$ is a bounded operator on $H^{2}\left(  \mathbb{C}%
^{+}\right)  $ for any $x\in\mathbb{R}$ and $t>0$, which is of course of a
particular importance for validation of the IST. A weaker versions of this
theorem (stated in different terms)\ was proven in \cite{GruRyPAMS13} (which
in turn improved \cite{Ryb10}).
\end{remark}

\section{Main Results\label{main section}}

In this section we finally state and prove our main results. With all the
preparations done in the previous sections, the actual proof will be quite short.

Note that while the interest to well-posedness of integrable systems has been
generated by the progress in soliton theory, well-posedness issues are
typically approached by means of PDEs techniques \cite{Tao06} (norm estimates,
etc.) and the IST is not usually employed. In soliton theory, in turn,
well-posedness is commonly assumed (frequently even by default) and one
applies the IST method to study the unique solution to (\ref{KdV}) or any
other integrable system. The paper \cite{KapTop06} represents a rather rare
example where the complete integrability of (\ref{KdV}) with periodic initial
data was used in a crucial way to prove some subtle well-posedness results for
irregular $q$ which are not accessible by harmonic analysis means. In our case
neither a priori well-posedness nor IST are readily available and we have to
deal with both at the same time.

Solutions of the KdV can be understood in a number of different ways
\cite{Tao06} (classical, strong, weak, etc.) resulting in a variety of
different well-posedness results.

\begin{definition}
[Natural solution]\label{WP} We call $q\left(  x,t\right)  $ a global natural
solution to (\ref{KdV}) if for any sequence of $C_{0}^{\infty}\left(
\mathbb{R}\right)  $ potentials $\left\{  q_{n}\left(  x\right)  \right\}  $
converging to $q\left(  x\right)  $ in $H_{\operatorname*{loc}}^{-1}\left(
\mathbb{R}\right)  $, the corresponding sequence of (classical) solutions
$\left\{  q_{n}\left(  x,t\right)  \right\}  $ to (\ref{KdV}) with initial
data $q_{n}\left(  x\right)  $ converges to $q\left(  x,t\right)  $ for any
$t>0$ uniformly in $x$ on compacts of $\mathbb{R}$.
\end{definition}

Our definition is a stronger version of that in \cite{KapTop06}. It also looks
quite natural from the computational and physical point of view. Another
feature of Definition \ref{WP} is that existence implies
uniqueness\footnote{and certain continuous dependence on the initial data
which we don't discuss here.}

\begin{theorem}
[Main Theorem]\label{MainThm} Assume that the initial data $q$ in (\ref{KdV})
is subject to Hypothesis \ref{hyp1.1}. Then the Cauchy problem (\ref{KdV}) has
a global natural solution $q(x,t)$ (Definition \ref{WP}) given by%
\begin{equation}
q(x,t)=-2\partial_{x}^{2}\log\det\left(  1+\mathbb{H}(x,t)\right)  ,
\label{det_form}%
\end{equation}
where $\mathbb{H}(x,t)$ is the IST Hankel operator associated with $q$
(Definition \ref{def IST HO}). The solution $q(x,t)$ has no singularities and
admits a meromorphic continuation $q(z,t)$ to the whole $\mathbb{C}$ with no
poles in parabolic domains%
\begin{equation}
D\left(  \delta,t\right)  :=\left\{  z:\frac{\operatorname{Im}^{2}z}%
{12}<\delta\operatorname{Re}z-\delta^{2}+\frac{\sqrt{\delta t}}{4}\log\frac
{t}{\delta^{3}}\right\}  \label{dom}%
\end{equation}
for any $t,\delta>0$.
\end{theorem}

\begin{proof}
Let $\left\{  q_{n}\left(  x\right)  \right\}  $ be any real $C_{0}^{\infty
}\left(  \mathbb{R}\right)  $ sequence converging to $q\left(  x\right)  $ in
$H_{\operatorname*{loc}}^{-1}\left(  \mathbb{R}\right)  $. Without loss of
generality we may assume that $q_{n}\left(  x\right)  \,$is of form
(\ref{Miura transform}). The problem (\ref{KdV}) with initial data
$q_{n}\left(  x\right)  $ is classical and its (unique) classical solution
$q_{n}\left(  x,t\right)  $ and can be computed by the Dyson formula%
\[
q_{n}\left(  x,t\right)  =-2\partial_{x}^{2}\log\det\left(  \mathbb{I}%
+\mathbb{H}_{n}(x,t)\right)  ,
\]
where $\mathbb{H}_{n}(x,t)$ is the IST Hankel operator corresponding to
$q_{n}$. By Theorem \ref{th9.5}, $q_{n}\left(  x,t\right)  $ is a meromorphic
function in $x$ on the entire complex plane. Consider the function $q\left(
x,t\right)  $\ given by (\ref{det_form}). By Theorem \ref{th9.5}, it is well
defined and entire in $x$ for any $t>0$. It remains to prove that $q\left(
x,t\right)  =\lim q_{n}\left(  x,t\right)  ,\ n\rightarrow\infty$, solves
(\ref{KdV}). Inserting $q=q_{n}+\Delta q_{n}$ into (\ref{KdV}) one gets%
\begin{align}
&  \partial_{t}q-6q\partial_{x}q+\partial_{x}^{3}q\label{rhs}\\
&  =\partial_{t}\Delta q_{n}+3\partial_{x}\left[  \left(  \Delta
q_{n}-2q\right)  \Delta q_{n}\right]  +\partial_{x}^{3}\Delta q_{n}.\nonumber
\end{align}
For $\Delta q_{n}$ we have (dropping subscript $x,t$)%
\[
\Delta q_{n}=-2\partial_{x}^{2}\log\det\left(  \mathbb{I}-\left(
\mathbb{I}+\mathbb{H}\right)  ^{-1}\left(  \mathbb{H-H}_{n}\right)  \right)  .
\]
It follows form (\ref{h=H}) and (\ref{Phi}) that for the symbol $\ \Delta
\Phi_{n}$ of $\mathbb{H-H}_{n}$ we have%
\begin{equation}
\Delta\Psi_{n}(k)=\frac{1}{2\pi i}\int_{\mathbb{R}}\xi_{x,t}(\lambda
+ih)\frac{R_{n}(\lambda+ih)-R(\lambda+ih)}{\lambda-k+ih}d\lambda.\label{delta}%
\end{equation}
But, by Theorem \ref{props of R} $R_{n}$ $\rightarrow R$ uniformly on compacts
in $\mathbb{C}^{+}$ as $n\rightarrow\infty$ and we can easily conclude that
$\partial_{t}^{m}\partial_{x}^{l}\left(  \mathbb{H-H}_{n}\right)  $ vanishes
in the trace norm as $n\rightarrow\infty$. Therefore $\partial_{t}^{m}%
\partial_{x}^{l}\Delta q_{n}\rightarrow0,n\rightarrow\infty$ and the right
hand side of (\ref{rhs}) vanishes.

By Theorem \ref{th9.5} $q\left(  x,t\right)  $ is meromorphic in $x$ for any
$t>0$ with all poles off the real line. It remains to show that it has no
poles in the domain (\ref{dom}). It follows from (\ref{h=H}) and (\ref{Phi})
that%
\begin{align*}
\left\Vert \mathbb{H}(z,t)\right\Vert  &  \leq\left\Vert \Phi_{z,t}\right\Vert
_{\infty}\\
&  \leq\frac{1}{2\pi h}\int\left\vert \xi_{z,t}(\lambda+ih)\right\vert
d\lambda.
\end{align*}
In virtue of (\ref{ksi})%
\[
\int_{\mathbb{R}}\left\vert \xi_{z,t}(\lambda+ih)\right\vert d\lambda
=\sqrt{\frac{\pi}{24ht}}\exp\left\{  8h^{3}t-2h\operatorname{Re}%
z+\frac{\operatorname{Im}^{2}z}{24ht}\right\}
\]
and hence for any $t,h>0.$%
\begin{equation}
\left\Vert \mathbb{H}(z,t)\right\Vert \leq\sqrt{\frac{1}{24\pi h^{3}t}}%
\exp\left\{  8h^{3}t-2h\operatorname{Re}z+\frac{\operatorname{Im}^{2}z}%
{24ht}\right\}  . \label{last estim}%
\end{equation}
The right hand side of (\ref{last estim}) is less than 1 if $z\in D\left(
\delta,t\right)  $ with $\delta=4h^{2}t$. Since $h$ is arbitrary $\delta$ is
also arbitrary
\end{proof}

In a weaker form for regular initial profiles Theorem \ref{MainThm} was proven
in recent \cite{GruRyPAMS13}. We conclude our paper with some discussions and corollaries.

\begin{remark}
Hypothesis \ref{hyp1.1} does not impose any decay assumption at $-\infty$ or
any type of pattern of behavior. Initial data $q\left(  x\right)  $ could be
unbounded at $-\infty$ or behave like white noise.
\end{remark}

\begin{remark}
We emphasize that our proof is based on limiting arguments and avoids dealing
directly with such common in the classical IST issues as the direct/inverse
scattering problem, time evolution of scattering quantities under the KdV
flow, etc. (see \cite{Hrynivetal2011} and the literature cited therein for
some results relevant to our singular initial data). This is the main
advantage of our approach and we only borrow the fact that the initial
condition is satisfied in $H_{\operatorname*{loc}}^{-1}\left(  \mathbb{R}%
\right)  $ sense \cite{KapPerryTopalov2005}.
\end{remark}

\begin{remark}
The estimate (\ref{s-number est}) means that the determinant in
(\ref{det_form}) rapidly converges. This fact, coupled with the recent
progress in computing Fredholm determinants \cite{Bornemann2010}, suggests
that (\ref{det_form}) could potentially be used for numerical evaluations.
\end{remark}

\begin{remark}
Theorem \ref{MainThm} says that any, no matter how rough, singular initial
profile $q\left(  x\right)  $ instantaneously evolves under the KdV flow into
a meromorphic function $q\left(  x,t\right)  $. This effect, also called
dispersive smoothing, has a long history. While being noticed long ago, its
rigorous proof took quit a bit of effort even for box shaped initial data
\cite{Murray(Cohen)78} (see also \cite{Zhou1997} for other integrable
systems). Note that our solutions are `dispersive', i.e. solutions which
disperse in time and do not have a soliton component.
\end{remark}

\begin{remark}
Since $q\left(  x,t\right)  $ is a meromorphic function on $\mathbb{C}$ for
any $t>0$, it is completely characterized by a countable number of time
dependent parameters. Viewing a pure soliton solution as a meromorphic
function of $x$ goes back to Kruskal. In \cite{Kruskal74} he initiated a study
of pole dynamics which has been quite active since then (see also
\cite{Air77}, \cite{Bona2009}, \cite{Segur2000}, \cite{GWeikard06} to mention
just four). In our soliton free situation we still in general have infinitely
many poles but their nature and behavior are unclear. We so far only know that
all poles are double, non-real, come in complex conjugate pairs, and stay away
from the time dependant domains $D\left(  \delta,t\right)  $ given by
(\ref{dom}). Besides this, some older general results \cite{Steinberg1969} say
that poles depend continuously on $t$ and cannot appear or disappear. We are
unaware of any relevant helpful results from the theory of Hankel operators
which would shed much light on the operator-valued function $\left(
\mathbb{I}+\mathbb{H}\left(  x,t\right)  \right)  ^{-1}$.
\end{remark}

\begin{remark}
As a meromorphic function $q\left(  x,t\right)  $ cannot vanish on a set of
positive Lebesgue measure for any $t>0$ unless $q\left(  x\right)  $ is
identically zero. This simple observation quickly recovers and improves on
many unique continuation results. E.g., $q\left(  x,t\right)  $ cannot have
compact support at two different moments unless it vanishes identically. This
result was first proven in \cite{Zhang92} assuming that $q\left(  x\right)  $
is absolutely continuous and short range. The techniques of \cite{Zhang92}
also rely on the IST\ and some Hardy space arguments.
\end{remark}

\begin{remark}
There is a large variety of determinant formulas similar to (\ref{det_form})
available in the literature. For instance, the substitution $q\left(
x,t\right)  =-2\partial_{x}^{2}\tau\left(  x,t\right)  $ (which goes back to
the seminal paper \cite{Hirota71}) is commonly used as an ansatz to reduce the
KdV equation to the so-called bilinear KdV which is advantageous in some
situations. Formulas like (\ref{det_form}) are particularly convenient for
describing classes of exact solutions (see, e.g. \cite{Ma05}) and $\tau\left(
x,t\right)  $ typically appears as a Wronskian. We also refer to
\cite{ErcMcKean90}, \cite{NovikovetalBook}, \cite{Popper84}, and
\cite{Venak86} for (\ref{det_form}) in the context of the Cauchy problem for
the KdV. Under our conditions on the initial data (\ref{det_form}) is new.
\end{remark}

\begin{remark}
It can be easily shown that $q\left(  x,t\right)  $ decays exponentially fast
in the region $x\geq Ct$ for any $C>0$.
\end{remark}

\section{Acknowledgement}

We are grateful to Rostislav Hryniv for valuable discussions.

\end{document}